\UseRawInputEncoding
\documentclass[12pt]{amsart}
\usepackage{setspace, amsmath, amsthm, amssymb, amsfonts, amscd, epic, graphicx, ulem, dsfont}
\usepackage[T1]{fontenc}
\usepackage{multirow}
\usepackage{bbm}
\usepackage{enumerate}

\makeatletter \@namedef{subjclassname@2010}{
  \textup{2020} Mathematics Subject Classification}
\makeatother

\newtheorem{thm}{Theorem}[section]

\newtheorem{cor}[thm]{Corollary}
\newtheorem{lem}[thm]{Lemma}
\newtheorem{pro}[thm]{Proposition}

\theoremstyle{remark}
\newtheorem*{rema}{Remark}

\theoremstyle{definition}

\newtheorem{exa}[thm]{\textbf{Example}}

\newcommand{\R}{\mathbb{R}}

\newcommand{\N}{\mathbb{N}}

\newcommand{\C}{\mathbb{C}}

\begin{document}

\title[Commutativity of symmetric operators]{On the commutativity of closed symmetric operators}
\author[S. Dehimi, M. H. Mortad and A. Bachir]{Souheyb Dehimi, Mohammed Hichem Mortad$^*$ and Ahmed Bachir}

\thanks{* Corresponding author.}
\date{}
\keywords{Normal operator; Closed operator; Symmetric operators;
Self-adjoint operators; positive operator; subnormal and hyponormal
operators; Fuglede-Putnam theorem; Hilbert space; Commutativity}

\subjclass[2010]{Primary 47B25. Secondary 47B15, 47A08.}

\address{(The first author) Department of Mathematics, Faculty of Mathematics and Informatics,
University of Mohamed El Bachir El Ibrahimi, Bordj Bou Arréridj,
El-Anasser 34030, Algeria.}

\email{souheyb.dehimi@univ-bba.dz, sohayb20091@gmail.com}

\address{(The corresponding author) Department of
Mathematics, University of Oran 1, Ahmed Ben Bella, B.P. 1524, El
Menouar, Oran 31000, Algeria.}

\email{mhmortad@gmail.com, mortad.hichem@univ-oran1.dz.}

\address{(The third author) Department of Mathematics, College of Science, King Khalid University, Abha, Saudi
Arabia.}

\email{abishr@kku.edu.sa, bachir\_ahmed@hotmail.com}

\begin{abstract}In this article, we give conditions guaranteeing the
commutativity of a bounded self-adjoint operator with an unbounded
closed symmetric operator.
\end{abstract}

\maketitle

\section{Essential background}

All operators considered here are linear but not necessarily
bounded. If an operator is bounded and everywhere defined, then it
belongs to $B(H)$ which is the algebra of all bounded linear
operators on $H$ (see \cite{Mortad-Oper-TH-BOOK-WSPC} for its
fundamental properties).

Most unbounded operators that we encounter are defined on a subspace
(called domain) of a Hilbert space. If the domain is dense, then we
say that the operator is densely defined. In such case, the adjoint
exists and is unique.

Let us recall a few basic definitions about non-necessarily bounded
operators. If $S$ and $T$ are two linear operators with domains
$D(S)$ and $D(T)$ respectively, then $T$ is said to be an extension
of $S$, written as $S\subset T$, if $D(S)\subset D(T)$ and $S$ and
$T$ coincide on $D(S)$.

An operator $T$ is called closed if its graph is closed in $H\oplus
H$. It is called closable if it has a closed extension. The smallest
closed extension of it is called its closure and it is denoted by
$\overline{T}$ (a standard result states that a densely defined $T$
is closable iff $T^*$ has a dense domain, and in which case
$\overline{T}=T^{**}$). If $T$ is closable, then
\[S\subset T\Rightarrow \overline{S}\subset
\overline{T}.\] If $T$ is densely defined, we say that $T$ is
self-adjoint when $T=T^*$; symmetric if $T\subset T^*$; normal if
$T$ is \textit{closed} and $TT^*=T^*T$.

The product $ST$ and the sum $S+T$ of two operators $S$ and $T$ are
defined in the usual fashion on the natural domains:

\[D(ST)=\{x\in D(T):~Tx\in D(S)\}\]
and
\[D(S+T)=D(S)\cap D(T).\]

In the event that $S$, $T$ and $ST$ are densely defined, then
\[T^*S^*\subset (ST)^*,\]
with the equality occurring when $S\in B(H)$. If $S+T$ is densely
defined, then
\[S^*+T^*\subset (S+T)^*\]
with the equality occurring when $S\in B(H)$.

Let $T$ be a linear operator (possibly unbounded) with domain $D(T)$
and let $B\in B(H)$. Say that $B$ commutes with $T$ if
\[BT\subset TB.\]
In other words, this means that $D(T)\subset D(TB)$ and
\[BTx=TBx,~\forall x\in D(T).\]

If a symmetric operator $T$ is such that $\langle Tx,x\rangle\geq0$
for all $x\in D(T)$, then we say that $T$ is positive, and we write
$T\geq0$. When $T$ is self-adjoint and $T\geq0$, then we can define
its unique positive self-adjoint square root, which we denote by
$\sqrt T$.

If $T$ is densely defined and closed, then $T^*T$ (and $TT^*$) is
self-adjoint and positive (a celebrated result due to von-Neumann,
see e.g. \cite{SCHMUDG-book-2012}). So, when $T$ is closed then
$T^*T$ is self-adjoint and positive whereby it is legitimate to
define its square root. The unique positive self-adjoint square root
of $T^*T$ is denoted by $|T|$. It is customary to call it the
absolute value or modulus of $T$. If $T$ is closed, then (see e.g.
Lemma 7.1 in \cite{SCHMUDG-book-2012})
\[D(T)=D(|T|)\text{ and } \|Tx\|=\||T|x\|,~\forall x\in D(T).\]

Next, we recall some definitions of unbounded non-normal operators.
A densely defined operator $A$ with domain $D(A)$ is called
hyponormal if
\[D(A)\subset D(A^*)\text{ and } \|A^*x\|\leq\|Ax\|,~\forall x\in D(A).\]

A densely defined linear operator $A$ with domain $D(A)\subset H$,
is said to be subnormal when there are a Hilbert space $K$ with
$H\subset K$, and a normal operator $N$ with $D(N)\subset K$ such
that
\[D(A)\subset D(N)\text{ and } Ax=Nx \text{ for all } x\in D(A).\]

\section{Some applications to the commutativity of self-adjoint operators}

In \cite{Alb-Spain}, \cite{Dehimi-Mortad-INVERT},
\cite{Gustafson-Mortad-I}, \cite{Gustafson-Mortad-II},
\cite{Jiang-product-hypoormal}, \cite{Jung-Mortad-Stochel},
\cite{MHM1}, \cite{Mortad-Thesis-Edinburgh-2003}, \cite{MHM7},
\cite{Mortad-FUG-PUT SURVEY BOOK}, \cite{Mortad-cex-BOOK}, and
\cite{Reh}, the question of the self-adjointness of the normal
product of two self-adjoint operators was tackled in different
settings (cf. \cite{Benali-Mortad}). In all cases, the commutativity
of the operators was reached.

Here, we deal with the similar question where the unbounded
(operator) factor is closed and symmetric which, and it is known, is
weaker than self-adjointness. More precisely, we show that if $B\in
B(H)$ is self-adjoint and $A$ is densely defined, closed and
symmetric, then $BA\subset AB$ given that $AB$ or $BA$ is e.g.
normal.

As in e.g. \cite{MHM1}, the use of the Fuglede-Putnam theorem is
primordial but just in the beginning of the proof. The desired
conclusions do not come as straightforwardly as when the operator
was self-adjoint. Thankfully, these technical difficulties have been
overcome.

Now, recall the well-known Fuglede-Putnam theorem:

\begin{thm}\label{Fug-Put UNBD A BD}(\cite{FUG}, \cite{PUT}) If $A\in B(H)$ and if $M$
and $N$ are normal (non-necessarily bounded) operators, then
\[AN\subset MA\Longrightarrow AN^*\subset M^*A.\]
\end{thm}

There have been many generalizations of the Fuglede-Putnam theorem
since Fuglede's paper. However, most generalizations were devoted to
relaxing the normality assumption. Apparently, the first
generalization of the Fuglede theorem to an unbounded $A$ was
established in \cite{Nussbaum-1969}. Then the first generalization
involving unbounded operators of the Fuglede-\textit{Putnam} theorem
is:

\begin{thm}\label{Fug-Put-MORTAD-PAMS-2003} If $A$ is a closed and symmetric operator and if $N$ is an unbounded normal operator, then
\[AN\subset N^*A\Longrightarrow AN^*\subset NA\]
whenever $D(N)\subset D(A)$.
\end{thm}

The previous result was established in \cite{MHM1} under the
assumption of the self-adjointness of $A$. However, and by
scrutinizing its proof in \cite{MHM1} or
\cite{Mortad-Thesis-Edinburgh-2003}, it is seen that only the
closedness and the symmetricity of $A$ were needed. This key
observation is all what one needs to start the proof of some of the
results below.

Let us also recall some perhaps known auxiliary results (cf. Lemmata
2.1 \& 2.2 in \cite{Gustafson-MZ-positive-PROD-1968}). See also
\cite{Meziane-Mortad-I} for the case of normality.

\begin{lem}\label{product}(\cite{Jung-Mortad-Stochel})
Let $A$ and $B$ be self-adjoint operators. Assume that $B\in B(H)$
and $BA \subseteq AB$. Then the following assertions hold:
\begin{enumerate}
\item[(i)]  $AB$ is a self-adjoint operator and $AB=\overline{BA}$,
\item[(ii)] if $A$ and $B$ are positive so is $AB$.
\end{enumerate}
\end{lem}

We shall also have need for the following result:

\begin{lem}\label{PAMS 03 LEMMA squr rt BA subset AB }
Let $B\in B(H)$ be self-adjoint. If $BA\subset AB$ where $A$ is
closed, then $f(B)A\subset Af(B)$ for any continuous function $f$ on
$\sigma(B)$. In particular, and if $B$ is positive, then $\sqrt
BA\subset A\sqrt B$.
\end{lem}

\begin{rema}
In fact, the previous lemma was shown in (\cite{MHM1}, Proposition
1) under the assumption "$A$ being unbounded and self-adjoint", but
by looking closely at its proof, we see that only the closedness of
$A$ was needed (cf. \cite{Bernau JAusMS-1968-square root} and
\cite{Jablonski et al 2014}).
\end{rema}

We are now ready to state and prove the first result of this
section.

\begin{thm}\label{BDM-THM-AB self-adjoint}
Let $A$ be an unbounded closed and symmetric operator with domain
$D(A)$, and let $B\in B(H)$ be positive. If $AB$ is normal, then
$BA\subset AB$, and so $AB$ is self-adjoint. Also, $\overline{BA}$
is self-adjoint.

Besides, $B|A|\subset |A|B$, and so $|A|B$ is self-adjoint and
positive. Moreover, $|A|B=\overline{B|A|}$.
\end{thm}

\begin{proof}Since $B\in B(H)$ is self-adjoint, we have
$(BA)^*=A^*B$ and $BA^*\subset (AB)^*$. Now, write
\[B(AB)=BAB\subset BA^*B\subset (AB)^*B.\]
Since $AB$ and $(AB)^*$ are both normal, the Fuglede-Putnam theorem
applies and gives
\[B(AB)^*\subset (AB)^{**}B=\overline{AB}B=AB^2,\]
i.e.
\[B^2A\subset B^2A^*\subset B(AB)^*\subset AB^2.\]
Since $A$ is closed and $B\in B(H)$ is positive, Lemma \ref{PAMS 03
LEMMA squr rt BA subset AB } gives
\[BA\subset AB.\]
To show that $AB$ is self-adjoint, we proceed as follows: Observe
that
\[BA\subset BA^*\subset (AB)^*.\]
Since we also have $BA\subset AB$, we now know that
\[BAx=ABx=(AB)^*x\]
for all $x\in D(A)$. This says that $AB$ and $(AB)^*$ coincide on
$D(A)$. Denoting the restrictions of the latter operators to $D(A)$
by $T$ and $S$ respectively, it is seen that
\[T-S\subset 0, ~T\subset AB,\text{ and }S\subset (AB)^*.\]
Hence
\[(AB)^*-AB\subset T^*-S^*\subset (T-S)^*=0.\]
Since $D(AB)=D[(AB)^*]$ thanks to the normality of $AB$, it ensues
that $AB=(AB)^*$, that is, $AB$ is self-adjoint.

Now, we show that $\overline{BA}$ is self-adjoint. First, we show
that $\overline{BA}$ is normal. Clearly $BA^*\subset A^*B$ for we
already know that $BA\subset AB$. Hence
\[BA^*A\subset A^*BA\subset A^*AB.\]
Therefore
\[\overline{BA}(\overline{BA})^*=ABA^*B\subset AA^*B^2\]
and
\[(\overline{BA})^*\overline{BA}=A^*BAB\subset A^*AB^2.\]
By Lemma \ref{product}, it is seen that both of $AA^*B^2$ and
$AA^*B^2$ are self-adjoint. By the maximality of self-adjoint
operators, it ensues that
\[\overline{BA}(\overline{BA})^*=AA^*B^2\text{ and }(\overline{BA})^*\overline{BA}=A^*AB^2.\]
Since $AB$ is self-adjoint, $(AB)^2$ is self-adjoint. But
\[(AB)^2=ABAB\subset AA^*B^2\]
and so $(AB)^2=AA^*B^2$. Similarly,
\[ABAB\subset A^*BAB\subset A^*AB^2\]
or $(AB)^2=A^*AB^2$. Therefore, we have shown that
\[(\overline{BA})^*\overline{BA}=\overline{BA}(\overline{BA})^*.\]
In other words, $\overline{BA}$ is normal.

To infer that $\overline{BA}$ is self-adjoint, observe that
$BA\subset AB$ gives $\overline{BA}\subset AB$, but because normal
operators are maximally normal, we obtain $\overline{BA}=AB$, from
which we derive the self-adjointness of $\overline{BA}$.

To show the last claim of the theorem, consider again $BA^*A\subset
A^*AB$. So, $B|A|\subset |A|B$ by the spectral theorem say. Since
$B\geq0$, Lemma \ref{product} gives the self-adjointness and the
positivity of $|A|B$, as well as $|A|B=\overline{B|A|}$. This
completes the proof.
\end{proof}

\begin{rema}
Under the assumptions of the preceding theorem (by consulting
\cite{Boucif-Dehimi-Mortad}), we have:
\[|AB|=|\overline{BA}|=|A|B=\overline{B|A|}.\]
\end{rema}

\begin{cor}\label{23/02/2022 ...COR}
Let $A$ be an unbounded closed and symmetric operator and let $B\in
B(H)$ be positive. Suppose that $AB$ is normal. Then
\[BA \text{ is closed} \Longrightarrow A\text{ is self-adjoint}.\]
In particular, if $B$ is invertible, then $A$ is self-adjoint.
\end{cor}

\begin{proof}By Theorem \ref{BDM-THM-AB self-adjoint}, $\overline{BA}$ is
self-adjoint and $\overline{BA}=AB$. Hence
\[BA^*\subset (AB)^*=(\overline{BA})^*=\overline{BA}.\]
So, when $BA$ is closed, $BA^*\subset BA$. Therefore, $D(A^*)\subset
D(A)$, and so $D(A)=D(A^*)$. Thus, $A$ is self-adjoint, as required.
\end{proof}

\begin{cor}
Let $A$ be an unbounded closed and symmetric operator with domain
$D(A)$, and let $B\in B(H)$ be positive. If $BA^*$ is normal, then
$BA\subset AB$, and so $BA^*$ is self-adjoint.
\end{cor}

\begin{proof}Since $BA^*$ is normal, so is $(BA^*)^*=AB$. To obtain
the desired conclusion, one just need to apply Theorem
\ref{BDM-THM-AB self-adjoint}.
\end{proof}

The case of the normality of $BA$ was unexpectedly trickier. After a
few attempts, we have been able to show the result.

\begin{thm}\label{BDM-THM-BA self-adjoint}
Let $A$ be an unbounded closed and symmetric operator with domain
$D(A)$, and let $B\in B(H)$ be self-adjoint. Assume $BA$ is normal.
Then $A$ is necessarily self-adjoint.

If we further assume that $B$ is positive, then $BA$ becomes
self-adjoint and $BA=AB$.
\end{thm}

We are now ready to show Theorem \ref{BDM-THM-BA self-adjoint}.

\begin{proof}First, recall that since $BA$ is normal, $BA$ is closed
and $D(BA)=D[(BA)^*]$.

Write
\[A(BA)\subset A^*BA=(BA)^*A.\]
Since $BA$ is normal and $D(BA)=D(A)$, Theorem
\ref{Fug-Put-MORTAD-PAMS-2003} is applicable and it gives
\[A(BA)^*\subset (BA)^{**}A=\overline{BA}A=BA^2,\]
i.e. $AA^*B\subset BA^2$. Since $A$ is symmetric, we may push the
previous inclusion to further obtain $AA^*B\subset BAA^*$, that is,
$|A^*|^2B\subset B|A^*|^2$.

Next, we claim that $B|A^*|$ is closed too. To see that, observe
that as $B\in B(H)$, then $(BA)^*=A^*B$. Hence
$\overline{BA}=(A^*B)^*$ or $BA=(A^*B)^*$ because $BA$ is already
closed. By Lemma 11 in \cite{CG}, the last equation is equivalent to
$(|A^*|B)^*=B|A^*|$ which gives the closedness of $B|A^*|$ as
needed.

Now, we have
\[B|A^*|(B|A^*|)^*=B|A^*|^2B\subset B^2|A^*|^2.\]
It then follows by Corollary 1 in \cite{DevNussbaum-von-Neumann}
that
\[B|A^*|(B|A^*|)^*=B^2|A^*|^2\]
for $B|A^*|(B|A^*|)^*$, $B^2$, and $|A^*|^2$ are all self-adjoint.
The self-adjointness of $B|A^*|(B|A^*|)^*$ also implies that
$B^2|A^*|^2$ is self-adjoint as well, i.e.
\[B^2|A^*|^2=(B^2|A^*|^2)^*=|A^*|^2B^2.\]
In particular, $B^2|A^*|^2$ is closed. So, Proposition 3.7 in
\cite{Dehimi-Mortad-INVERT} implies that $B|A^*|^2$ is closed.

The next step is to show that $B|A^*|^2$ is normal. As
$|A^*|^2B\subset B|A^*|^2$, it ensues that
\[B|A^*|^2(B|A^*|^2)^*=B|A^*|^4B\subset B^2|A^*|^4\]
and
\[(B|A^*|^2)^*B|A^*|^2=|A^*|^2B^2|A^*|^2\subset B^2|A^*|^4.\]
Since $B|A^*|^2(B|A^*|^2)^*$, $(B|A^*|^2)^*B|A^*|^2$, $B^2$, and
$|A^*|^2$ are all self-adjoint, Corollary 1 in
\cite{DevNussbaum-von-Neumann} yields
\[B|A^*|^2(B|A^*|^2)^*=(B|A^*|^2)^*B|A^*|^2~~(=B^2|A^*|^4).\]
Therefore, $B|A^*|^2$ is normal. So, since $B\in B(H)$ is
self-adjoint and $|A^*|^2$ is self-adjoint and positive, it follows
by Theorem 1.1 in \cite{Gustafson-Mortad-II} that $B|A^*|^2$ is
self-adjoint and $B|A^*|^2=|A^*|^2B$.

By applying Theorem 10 in \cite{Bernau JAusMS-1968-square root}, it
is seen that
\[B|A^*|=|A^*|B\]
due to the self-adjointness and the positivity of $|A^*|$.

We now have all the necessary tools to establish the
self-adjointness of $A$. Indeed,
\begin{align*}
D(A^*)=D(|A^*|)=D(B|A^*|)&=D(|A^*|B)\\&=D(A^*B)=D[(BA)^*]=D(BA)=D(A).
\end{align*}
Thus, $A$ is self-adjoint as it is already symmetric.

Finally, when $B\in B(H)$ is positive and since $A$ is self-adjoint,
$(BA)^*=AB$ is normal. By Theorem \ref{BDM-THM-AB self-adjoint},
$AB$ is self-adjoint or $(BA)^*$ is self-adjoint. In other words,
\[BA=(BA)^*=AB,\]
and this marks the end of the proof.
\end{proof}

Generalizations to weaker classes than normality vary. Notice in
passing that in \cite{Dehimi-Mortad-INVERT}, the self-adjointness of
$BA$ was established for a positive $B\in B(H)$ and an unbounded
self-adjoint $A$ such that $BA$ is hyponormal and
$\sigma(BA)\neq\C$. The next result is of the same kind.

\begin{pro}Let $B\in B(H)$ be positive and let $A$ be a
densely defined closed symmetric operator. If $(AB)^*$ is subnormal
or if $BA^*$ is closed and subnormal, then $BA\subset AB$.

Moreover, if $A$ is self-adjoint, then $AB$ is self-adjoint.
Besides, $AB=\overline{BA}$.
\end{pro}

\begin{proof}The proof relies on a version of the Fuglede-Putnam theorem
obtained by J. Stochel in \cite{STO-asymm-Fuglede-Putnam-PAMS-2001}.
Write
\[B[(AB)^*]^*=B(AB)^{**}=BAB\subset BA^*B\subset (AB)^*B.\]
Since $(AB)^*$ is subnormal, Theorem 4.2 in
\cite{STO-asymm-Fuglede-Putnam-PAMS-2001} yields
\[B^2A\subset B^2A^*\subset B(AB)^*\subset (AB)^{**}B=AB^2.\]

The same inclusion is obtained in the event of the subnormality of
$BA^*$. Indeed, write
\[B(BA^*)^*=BAB\subset BA^*B.\]
Applying again Theorem 4.2 in
\cite{STO-asymm-Fuglede-Putnam-PAMS-2001} gives
\[B(BA^*)\subset (BA^*)^*B=AB^2.\]
Therefore, and as above, we obtain $B^2A\subset AB^2$.

Now, since $B\geq 0$ and $A$ is closed, it follows that $BA\subset
AB$.

Finally, when $A$ is self-adjoint, Lemma \ref{product} implies that
$AB$ is self-adjoint and $AB=\overline{BA}$, as needed.
\end{proof}

There are still more cases to investigate. As is known, if $N\in
B(H)$ is such that $N^2$ is normal, then $N$ need not be normal (cf.
\cite{Mortad-square-root-normal}). The same applies for the class of
self-adjoint operators.

The first attempted generalization is the following: Let $A,B$ be
two self-adjoint operators, where $B$ is positive, and such that
$(AB)^n$ is normal for some $n\in\N$ such that $n\geq2$. Does it
follow that $AB$ is self-adjoint?

The answer is negative even when $A$ and $B$ are $2\times 2$
matrices. This is seen next:

\begin{exa}Take
\[A=\left(
      \begin{array}{cc}
        0 & 1 \\
        1 & 0 \\
      \end{array}
    \right)
 \text{ and }B=\left(
                 \begin{array}{cc}
                   0 & 0 \\
                   0 & 1 \\
                 \end{array}
               \right).
 \]
Then $A$ is self-adjoint and $B$ is positive (it is even an
orthogonal projection). Also,
\[AB=\left(
      \begin{array}{cc}
        0 & 1 \\
        0 & 0 \\
      \end{array}
    \right)\text{ whilst }(AB)^n=\left(
                 \begin{array}{cc}
                   0 & 0 \\
                   0 & 0 \\
                 \end{array}
               \right)\]
    for all $n\geq2$. In other words, $AB$ is not self-adjoint while all
    $(AB)^n$, $n\geq2$, are patently self-adjoint.
\end{exa}

Let us pass to other possible generalizations.

\begin{pro}
Let $B\in B(H)$ be positive and let $A$ be a closed and symmetric
operator. Assume $AB^n$ is normal for a certain positive integer
$n\in\N$. Then
\begin{enumerate}
  \item $BA\subset AB$ (hence $BA$ is symmetric).
  \item If it is further assumed that $B$ is invertible, then $A$ is self-adjoint. Besides, all of
  $AB^{1/n}$ and $B^{1/n}A$ are self-adjoint for all $n\geq 1$.
\end{enumerate}
\end{pro}

\begin{proof}\hfill
\begin{enumerate}
  \item Since $B^n$ is positive for all $n$ and $AB^n$ is normal, it follows by Theorem \ref{BDM-THM-AB self-adjoint} that $AB^n$ is
  self-adjoint and $B^nA\subset AB^n$. By Lemma \ref{PAMS 03 LEMMA squr rt BA subset AB
  }, it is seen that $BA\subset AB$.
  \item Since $AB^n$ is normal and $B^nA$ is closed (as $B^n$ is
  invertible), Corollary \ref{23/02/2022 ...COR} yields the
  self-adjointness of $A$.

  Finally, since $BA\subset AB$ and $B$ is positive, it follows that
  $B^{1/n}A\subset AB^{1/n}$, from which we derive the
  self-adjointness of $AB^{1/n}$ and $\overline{B^{1/n}A}=B^{1/n}A$,
  as suggested.
\end{enumerate}
\end{proof}

Similarly, we have:

\begin{pro}\label{23/01/2022}
Let $B\in B(H)$ be positive and let $A$ be a closed and symmetric
operator. Assume that $B^nA$ is normal for some positive integer
$n\in\N$. Then $A$ and $BA$ are self-adjoint, and $BA=AB$.
\end{pro}

One of the tools to prove this result is:

\begin{lem}\label{BnA closed implies BA closed LEMMA}(Cf. Proposition 3.7
in \cite{Dehimi-Mortad-INVERT}) Let $B\in B(H)$ and let $A$ be an
arbitrary operator such that $B^nA$ is closed for some integer
$n\geq2$. Suppose further that $BA$ is closable. Then $BA$ is
closed.
\end{lem}

\begin{proof}Let $(x_p)$ be in $D(B^nA)$ and such that $x_p\to x$ and $BAx_p\to
y$. Since $B^{n-1}\in B(H)$, $B^nAx_p\to B^{n-1}y$. Since $B^nA$ is
closed, we obtain $x\in D(B^nA)=D(A)$. Since $BA\subset
\overline{BA}$ and $x\in D(BA)$, we have
\[BAx=\overline{BA}x=\lim_{p\to \infty}BAx_p=y\]
by the definition of the closure of an operator. We have therefore
shown that $BA$ is closed, as wished.
\end{proof}

Now, we show Proposition \ref{23/01/2022}.

\begin{proof} Since $B^n$ is positive, Theorem \ref{BDM-THM-BA
self-adjoint} gives both the self-adjointness of $A$ and $B^nA$.
Moreover, $B^nA=AB^n$. Using Lemma \ref{PAMS 03 LEMMA squr rt BA
subset AB } or else, we get $BA\subset AB$ (only the inclusion
suffices to finish the proof). The equation $B^nA=AB^n$ contains the
closedness of $B^nA$ which, by a glance at Lemma \ref{BnA closed
implies BA closed LEMMA}, yields $BA=AB$ by consulting Lemma
\ref{product}.
\end{proof}

\end{document}